\documentclass[10pt, reqno]{amsart}

\usepackage{latexsym,enumerate}
\usepackage{amsmath,amsthm,amsopn,amstext,amscd,amsfonts,amssymb,url}
\usepackage{color}
\numberwithin{equation}{section}

\newtheorem{theorem}{Theorem}
\newtheorem{lemma}{Lemma}

\newtheorem{remark}{Remark}
\newtheorem{definition}{Definition}

\numberwithin{theorem}{section}
\numberwithin{corollary}{section}
\numberwithin{lemma}{section}
\numberwithin{definition}{section}
\numberwithin{proposition}{section}
\numberwithin{remark}{section}
\numberwithin{example}{section}

\def\divi{\hbox{\rm div\,}}
\def\R{\mathbb R}

\def\N{\mathbb N}

\def\sign{{\rm sign}\,}

\def\qed{{\quad\rule{2mm}{2mm}\medskip}}

\begin{document}

\title[Singular elliptic  equations]{Singular elliptic  equations
having a gradient term with natural growth}

\begin{abstract}
We study a class of Dirichlet boundary value problems whose prototype is 
\begin{equation}\label{1.2abs}
 \left\{\begin{array}{ll}
 -\Delta_p u =h(u)|\nabla u|^p+u^{q-1}+f(x)\, &\quad\hbox{in } \ \Omega\,,\\
u\ge 0\,,&{\quad\hbox{in } \ \Omega}\\
u = 0\,&\quad\hbox{on }\partial \Omega\,,\end{array}\right.
\end{equation}
where   $\Omega$ an open bounded subset of $\R^N$,  $0<q<1$, $1<p<N$, $h$ is a continuous function and $f$ belongs to a suitable Lebesgue space.  The main features of this problem are the presence of a singular term and a first order term with natural growth in the gradient.
A priori estimates and existence results are proved depending on the summability of the datum $f$.
\end{abstract}

\author[A. Ferone, A. Mercaldo, S. Segura de Le\'on]
{A. Ferone, A. Mercaldo and S. Segura de Le\'on}
\subjclass{35B25, 35J25}
\keywords{Existence, Singular elliptic equations, a priori estimates, gradient term with natural growth}

\address{Adele Ferone
\hfill \break\indent Dipartimento di Matematica e Fisica,
\hfill\break\indent Universit\`a degli Studi della Campania  ``L.Vanvitelli"
\hfill\break\indent Viale Lincoln 5, 81100 Caserta, Italy.}
\email{{\tt adele.ferone@unicampania.it}}

\address{Anna Mercaldo
\hfill \break\indent Dipartimento di Matematica e Applicazioni
``R.Caccioppoli",
\hfill\break\indent Universit\`a degli Studi di Napoli
  Federico II,
\hfill\break\indent Complesso Monte S. Angelo, Via
Cintia, 80126  Napoli, Italy.} \email{{\tt mercaldo@unina.it}}

\address{Sergio Segura de Le\'on
\hfill \break\indent Departament d'An\`alisi Matem\`atica,
Universitat de Val\`encia, \hfill\break\indent Dr. Moliner 50,
46100 Burjassot, Val\`encia, Spain.} \email{{\tt
sergio.segura@uv.es}}

\maketitle

\section{Introduction}\label{s1}

In the present paper we study the  existence of a nonnegative solution $u$   for a nonlinear elliptic equation having both a zero order term which tends to infinity at $u=0$  and a first order term which has a natural growth in the gradient of $u$.  More precisely, this paper concerns with problems of the type
\begin{equation}\label{p_intro}
 \left\{\begin{array}{ll}
 -\divi(\pmb a(x, u, \nabla u)) =b(x ,u, \nabla u)+g(x,u)+f(x)\,, &\quad\hbox{in } \ \Omega\\[3mm]
u = 0\,,&\quad\hbox{on }\partial \Omega\,,\end{array}\right.
\end{equation}
Here $\Omega$ is an open bounded subset of $\R^N$, $N\ge 2$, 
$-\divi(\pmb a(x, u, \nabla u))$ is a Leray-Lions operator defined on $W^{1,p}_0(\Omega)$, $b(x ,u, \nabla u)$ is a nonlinear term which grows like $|\nabla u|^p$ and more precisely satisfies
$$
|b(x, s, z )|\le h(s) |z|^p\,,
$$
for a continuous function $h:\R\rightarrow \R$. Moreover,  $g(x,u)$ is a singular term  at $u=0$, that is  
$$
0\le g(x, s)\le   \Lambda s^{q-1}\,, \quad \Lambda >0\,, \quad 0<q<1 \,.
$$
Finally the datum $f$ belongs to a suitable Lebesgue space. 

The existence of a weak solution to problem \eqref{p_intro} when the singular term   $g(x,u) $ does not appear has been investigated by many authors  starting by the 80s.
 In the papers \cite{BMP2, BMP4} the existence of bounded solutions is proved when the problem \eqref{p_intro} has   also a (non singular) zero-order absorption term, while in papers \cite{BMP1} a more general class of equations, which have a gradient term satisfying suitable sign conditions, is considered and existence of unbounded solutions has been proved.

Existence and nonexistence of unbounded solutions to problem \eqref{p_intro} having a reaction gradient terms with natural quadratic growth  satisfying a further suitable condition have been faced in \cite{FM1, BST, Segura}  by using test functions that simulate the Cole–Hopf transformation. The more general case, where  natural growth different of the quadratic one, is treated in  \cite{FM2, PS}.
 In \cite{PS} optimal conditions on the growth of $h$ at infinity to ensure that, given $f$ with a certain summability, problem \eqref{p_intro} admits a solution are given. Similar results in this order of ideas can be found in \cite{Porretta}.

The existence of nonnegative solutions for semilinear  second order partial differential problems singular at $u=0$, without first order term, is also a classical problem which has been considered by several authors since 70's. 
In \cite{CRT}, it is shown the existence and uniqueness of a nonnegative solution in a case where the equation is not written in a divergence form and the solution is a classical solution, i.e. it is in $C^2(\Omega)\cap C (\overline\Omega)$ which is strictly positive in $\Omega$. In \cite{CP} it is considered the case $g (x, s) = 1/s^\gamma + (\lambda s)^\alpha$ with $\gamma , \lambda, \alpha >
0$ and existence and nonexistence results for classical solutions are proved (see also \cite{Stuart}). 
Looking for a nonnegative solution in a Sobolev space, the problem has been considered in \cite{BO}, where it is studied the case $g(x, u)=f (x)/s^\gamma$, with $\gamma>0$, $f\ge 0$ not identically zero and belonging to a suitable Lebesgue space. In this paper existence, uniqueness and regularity of
a distributional solution, strictly positive in $\Omega$ is proved.  The proofs of these results are manly based on the fact that $g(x, s)$ is a nonincreasing function in the
variable $s$ and on the use of strong maximum principle. Uniqueness and comparison results for this type of solution has been
proved in \cite{BCD} and, by using symmetrization techniques, in \cite{BCT}. In order to avoid the use of strong maximum principle and monotonicity assumption on $g(x, \cdot)$, a new definition of nonnegative
solutions have been provided and existence, stability and uniqueness results for these notions are proved   in  \cite{ AGM1, AGM2, AGM3}. 

\noindent Further contributions to semilinear elliptic equations having this type of singularity are contained for example
in \cite{AB, AMM, B, BCroce, CD, CGS, CES, OlP, OrP}.

A very few results are known about existence of solutions which changes its sign and a first paper in this direction is  \cite{CDM}, where the authors show that if the “singular term” goes to infinity at zero faster than $1/|u|$ then only nonnegative solutions are possible, while in the other case nonpositive solution or even solutions changing the sign are possible. In \cite{FMS1} the solution is defined as a minimum point for a suitable functional and the definition of solution given by the authors uses test functions which vanish at $u = 0$ and thus the equation is satisfied in $\Omega \setminus \{u = 0\}$. It is also proved the uniqueness for nonnegative solutions when g(x, .) is decreasing.

The effects of the presence of two singularities, both in a gradient term having natural growth and in a zero order term, has been addressed in \cite{O}.  In this paper the function $h$ is summable on $\R$ and the singular zero order term involves the datum: $g(x,u)=f(x)/u^\gamma$. These assumptions allow the study of the equation with $L^1$--data.

The novelty of  this paper consists in analyzing the effects of a gradient term having natural growth and a singular term of the type $g(x, u)=1/u^\gamma$. Following \cite{PS}, the existence of  nonnegative solutions to problem \eqref{p_intro} is proven depending on  the  behaviour of $h$ at $+\infty$  (we point out that we do not assume $h\in L^1(\R)$) and the summability of the datum $f$. 

Our approach is based on  a priori estimates for weak solutions to a sequence of approximating problems. The  proof of these a priori estimates is obtained by adapting the classical approach to study these type of equations made by a change of unknown of Cole--Hopf type given by Lemma \ref{canc}. These a priori estimates allow to deduce the existence of a limit function $u$ such that the approximate solutions $u_n$ and their gradients $\nabla u_n$ converge to $u$ and  $\nabla u$ respectively. A procedure of passage to the limit permits to prove that such a function $u$ is, indeed, a solution of problem \eqref{p_intro}. 
The main difficulties in proving firstly the a priori estimates and then in passing to limit in  the approximating problems are due to the presence of the singular term $g(x, u)$ and the necessity to prove that it can really be bounded

\noindent In this paper, we consider three different types of summability for a datum $f\in L^m(\Omega)$:  (a) $m> N/p$, (b) $m= N/p$ and (c) $\frac{Np}{Np-N+p}\le m<\frac Np$, which will be analyzed separately into three existence results, one for each. Our main results are given by Theorem \ref{main}, Theorem \ref{mainlimit} and Theorem \ref{main_unbounded} and proven in Section 3, Section 4 and Section 5, respectively.

\section{Notation, assumptions and preliminary results}\label{assumptions}

Throughout this paper, $\Omega$ stands for an open bounded set of $\R^N$, with $N\ge 2$. The Lebesgue measure of $E\subset \Omega$ will be denoted by $|E|$.

\noindent On the other hand, the positive and negative part of a function $u$ is denoted by $u_+$ and $u_-$, respectively. Moreover, we denote
$$
\{|u|\ge \delta \}=\{x\in \Omega: \>|u(x)|\ge \delta \}\,,
$$
for any $\delta >0$.

\noindent In what follows, we will also consider two auxiliary functions.
For any $s\in \R$ and any $k>0$ we define
\begin{equation}\label{gk}
 G_k(s)=(|s|-k)_+\sign(s) \,,
\end{equation}
\begin{equation}\label{tk}
T_k(s) = \max\{-k, \min\{s, k\}\} \,.
\end{equation}

\noindent From now on, we will denote by $C$ a positive constant that only depends on the data, not  on $n$ and that  may change from line to line.

\subsection{Assumptions}

The aim of this subsection is to give the hypotheses on the data of problem \eqref{p_intro} which we make in the whole paper.
We also introduce the notion of weak solution which we use.

As pointed out in the Introduction we study solutions to the following singular nonlinear problem
 \begin{equation}\label{P}
 \left\{\begin{array}{ll}
 -\divi(\pmb a(x, u,\nabla u) )=b(x ,u, \nabla u)+g(x,u)+f(x)\,, &\quad\hbox{in } \ \Omega\\[3mm]
u\ge 0\,,&\quad\hbox{in } \ \Omega\\[3mm]
u = 0\,,&\quad\hbox{on }\partial \Omega\,.\end{array}\right.
\end{equation}
 We assume that, for some $1<p<N$,
 \begin{align*}
 &\pmb{a}\>:\>\Omega\times\mathbb{R} \times\mathbb{R}^{N}\rightarrow \mathbb{R}^{N}\\
 &b :\Omega\times\mathbb{R} \times\mathbb{R}^{N}\rightarrow \mathbb{R}\\
&g\>:\>\Omega\times (\R\setminus \{0\}) \longrightarrow [0,+\infty)
\end{align*}
are   Carath\'{e}odory functions which satisfy
the growth conditions
\begin{equation}\label{crescitaa}
|\pmb{a}(x, s, z )|\le a_0 |z|^{p-1}+a_1|s|^{p-1}+a_2  \, ,\quad a_0 , a_1, a_2 >0\,,
\end{equation}
\begin{equation}\label{crescitab}
|b(x, s, z )|\le h(s) |z|^p\,,
\end{equation}
where  $h:\R\rightarrow \R$ is a continuous function, and
\begin{equation}\label{g_pos}
 0\le g(x, s)\le   \Lambda s^{q-1}\,, \quad \Lambda >0\,, \quad 0<q<1 \,.
\end{equation}
Moreover the function $\pmb{a}$ satisfies the ellipticity condition
\begin{equation}
\pmb{a}\left( x,  s, z \right) \cdot z \geq  \lambda\left\vert z\right\vert ^{p} \,, \qquad \lambda >0 \,,
 \label{ell}
\end{equation}
and the monotonicity condition
\begin{equation}
\left( \pmb{a}\left( x,  s, z \right) -\mathbf{a}\left( x, s, z' \right)
\right) \cdot \left( z -z' \right) >0\,,
\label{monotonia}
\end{equation}
 These hypotheses hold
 for every $z, z'\in \R^N$, with $z \neq z'$, for every $s\in \R$ and for almost every $x\in\Omega$.

\noindent Finally we assume
\begin{equation}\label{ipf_pos}
f\ge 0\, \quad \hbox{and} \quad f\in L^m(\Omega)\,,
\end{equation}\medskip
where $m$ will be specified later.
\begin{remark}\rm
It is worth remarking that no singularity occurs in the product $g(x,s)s^{1-q}$ since
\begin{equation}
  g(x,s)\,s^{1-q}\le \Lambda\,.
\end{equation}
Obviously the product $g(x,s)s$ has also no singularities.
\end{remark}
Now we give the definition of weak solution to problem \eqref{P}  whose existence is proved in Sections 3 - 5.
\begin{definition}\label{defsol}
A function $u\in W_0^{1,p}(\Omega)$ is a {\bf weak solution} to \eqref{P} if
\begin{equation}\label{4.00}
\frac{|\nabla u|^p}{|u|^{q}}\in L^1(\Omega)\,,
\end{equation}
\begin{equation}\label{b}
b(x, u, \nabla u)\in L^1(\Omega)\,,
\end{equation}
\begin{equation}\label{4.01_pos}
\int_\Omega g(x,u)v\,dx<+\infty\,, \qquad  \forall v\in W^{1,p}_0(\Omega)\cap L^\infty(\Omega) \,,
\end{equation}
and
\begin{multline}\label{4.0}
  \int_\Omega \pmb a(x, u,\nabla u)\cdot\nabla\varphi\, dx
    =
    \int_\Omega b(x,u, \nabla u) \varphi\, dx
    +\int_\Omega g(x,u) \varphi+\int_\Omega f \varphi\,.
\end{multline}
for every  $\varphi\in W^{1,p}(\Omega)\cap L^\infty(\Omega)$.
\end{definition}

\subsection{Auxiliary functions}
In this section we recall some well-known facts which are tools to address the study of quasi-linear elliptic equations with natural growth in the gradient.

\noindent Denote
\begin{equation}\label{H}
H(s)=\frac 1\lambda\int_0^sh(\sigma)\, d\sigma\,,
\end{equation}
\begin{equation}\label{fi}
\Phi(s)= \int_0^s e^\frac{|H(\sigma)|}{p-1}\, d\sigma\,.
\end{equation}
These auxiliary functions are used throughout the whole paper. A simple remark is in order: every function $u$ satisfies
\begin{equation}\label{fiu}
|\Phi(u)|=\left|\int_0^ue^\frac{|H(\sigma)|}{p-1}\, d\sigma\right|\ge |u|
\end{equation}
\begin{equation}\label{fisuu}
|\Phi(u)|\le |u| e^{\frac{|H(u)|}{p-1}}
\end{equation}
and
\begin{equation}\label{figrad}
|\nabla\Phi(u)|=e^\frac{|H(u)|}{p-1}|\nabla u|\ge |\nabla u|\,.
\end{equation}
Therefore, if $\Phi(u)\in L^r(\Omega)$  or $\Phi(u)\in W_0^{1,r}(\Omega)$  for some $r\ge 1$,  then $u\in L^r(\Omega)$ or $u\in W_0^{1,r}(\Omega)$ respectively.

\subsection{Approximating problems}

We are concerned with proving that problem \eqref{P}  has at least a weak solution. We will prove this result by approximations. To do so, we consider the following problems
\begin{equation}\label{pd1}
 \left\{\begin{array}{ll}
 -{\rm div}(\pmb a(x,u_n,\nabla  u_n)) =b(x, u_n, \nabla u_n)+ g_n (x, u_n )+f_n(x)\,, &\hbox{in } \Omega\\[3mm]
 u_n\ge 0\,,&\hbox{in }  \Omega\\[3mm]
u_n = 0&\hbox{on }\partial \Omega\end{array}\right.
\end{equation}
where $g_n(x,s)=T_n(g(x,s))$ and $f_n=T_n(f)$.
For any fixed $n$, problem \eqref{pd1} exhibits at least a weak solution $u_n\in W^{1,p}_0(\Omega)\cap L^\infty(\Omega)$ as a consequence of the results of \cite{PS}, it is enough to take $b_0(x)=n$ in \cite[Theorem 1.1]{PS}.

Actually, we may straightly consider the datum $f$, without truncations, since by assumptions on its summability, the datum $f$ is always an element of the dual space of the Sobolev space $W^{1}_0(\Omega)$, $W^{-1,p'}(\Omega)$. Starting from the truncated data allows us to easily use a cancellation lemma. It is a consequence of a kind of change of unknown obtained multiplying the equation by a suitable exponential function of $u_n$ (see, for example, \cite{PS}). Since $u_n\in  L^\infty(\Omega)$, there is no need to worry about whether these exponential test functions can really be chosen.

\begin{lemma} \label{canc}(Cancellation lemma)
Let  $u_n  \in W^{1, p}_0 ( \Omega) \cap L^{\infty} ( \Omega ) $  be a weak solution to problem
(\ref{pd1}).
\begin{item}
\item (1) If  $v \in W^{1, p}_0 ( \Omega) $, then
\begin{multline*}
\int_{\Omega} e^{\hbox{sign\,}(v) H(u_n)} \pmb a (x, u_n,\nabla u_n) \cdot \nabla v\, dx\\
\le \int_{\Omega} e^{\hbox{sign\,}(v) H(u_n)} v g(x,u_n)  \, dx+ \int_{\Omega} e^{\hbox{sign\,}(v) H(u_n)} v f  \, dx.
\end{multline*}
\item (2) If  $\Psi$ is a locally Lipschitz continuous and nondecreasing
function such that $\Psi (0) = 0$, then
\[ \lambda\int_{\Omega} e^{|H (u_n)|} \Psi^{\prime}
(u_n)  |\nabla u_n|^pdx \le \int_{\Omega} e^{|H (u_n)|} \Psi (u_n) g(x,u_n)\,dx+ \int_{\Omega} e^{|H (u_n)|} \Psi (u_n)f\,dx\, . \]
\end{item}
\end{lemma}

We first apply this Lemma to check that the approximate solutions are nonnegative. To this end, we choose $v=-u_n^-$ in Lemma \ref{canc} (1) to get
\begin{multline*}
  \int_{\{u_n<0\}}e^{-H(u_n)}\, \pmb a(x,u_n,\nabla  u_n)\cdot \nabla u_n \\
  \le -\int_\Omega (e^{-H(u_n)}u_n^-)T_n(g (x, u_n ))-\int_\Omega (e^{-H(u_n)}u_n^-) f_n
\end{multline*}
By \eqref{ell}, since the right-hand side is nonpositive, we obtain
\[\lambda \int_{\{u_n<0\}}e^{-H(u_n)} |\nabla u_n|^p\le 0\]
from where $u_n\ge 0$ a.e. in $\Omega$ follows.

\section{Existence result for $m> N/p$}

The main result of this section concerns existence of nonnegative weak solutions to problem \eqref{P} when the datum $f$ is  an element of the Lebesgue space $L^m(\Omega)$, with $m> N/p$. It is given by the following theorem:
\begin{theorem}\label{main}
Assume \eqref{crescitaa}-\eqref{ipf_pos} with
$$
f\in L^m(\Omega)\,,\qquad m> \frac Np
$$
and
\begin{equation}\label{ipo0}
\lim_{s\to \pm \infty}\frac{e^{H(s)}}{(1+\Phi(s))^{p-1}}=0\,.
\end{equation}
Then problem \eqref{P}  has at least a weak solution $u$ such that  $\Phi(u)\in W_0^{1,p}(\Omega)\cap L^\infty(\Omega)$ and, consequently,  $u\in W_0^{1,p}(\Omega)\cap L^\infty(\Omega)$.
  \end{theorem}

We prove that the sequence of {\sl approximate solutions} $\{u_n\}_n$ to problem \eqref{pd1} satisfies some a priori estimates. By these estimates we  deduce that $u_n$, up to subsequences, converges to a function $u$ which we prove is the sought  weak solution.

\medskip

\subsection{A priori estimates when $m>N/p$}\label{s3dex1}

In this subsection we  prove that the sequence of {\sl approximate solutions} $\{u_n\}_n$ satisfies a priori estimates in $L^\infty(\Omega)$ and in $W^{1,p}_0(\Omega)$. As pointed out, by these estimates we  deduce that $u_n$ converges,  up to subsequences,  to a function $u$ which is the sought solution.

\begin{lemma}\label{apriori}
({\sl Estimates in $L^\infty(\Omega)$ and $W^{1,p}_0(\Omega)$}).  For any fixed $n\in \N$, let $u_n\in W^{1,p}_0(\Omega)\cap L^\infty(\Omega)$ be a weak solution to problem \eqref{pd1}. Under the assumptions of Theorem \ref{main}, the following estimates hold true:
\begin{equation}\label{infty}
\|u_n\|_{L^\infty(\Omega)}\le C_1\,,
\end{equation}
\begin{equation}\label{h01}
\|\nabla u_n\|_{L^p(\Omega)}\le C_2\,,
\end{equation}
where $C_1$, $C_2$ are positive constants which only depend on $|\Omega|$, $N$, $m$, $p$, $\|f\|_{L^m}$, $\lambda$,
but do  not depend on $n$.
\end{lemma}
\begin{proof} Most of the proof follows that of \cite[Proposition 3.1]{PS}. We insert it to highlight that the presence of the singular term does not affect the result.

For any fixed $k>0$
 consider  $\Psi(s)= G_k(\Phi(s))$ in Lemma \ref{canc} (2) with $g (x, u_n)$ replaced by $T_n(g (x, u_n)) $. By \eqref{ell}, \eqref{g_pos} and H\"older's inequality, we obtain
\begin{align}\label{est}
 &\lambda\int_\Omega|\nabla G_k(\Phi(u_n))|^p\, dx\\[5mm]
  & \le\int_\Omega T_n(g (x, u_n)) \, e^{H(u_n)} \,|G_k(\Phi(u_n))|\, dx
      +\int_\Omega f_n\,e^{H(u_n)} \,|G_k(\Phi(u_n))|\, dx\notag\\[5mm]
     & \notag\le\Lambda \int_\Omega  u_n^{q-1}\, e^{H(u_n)} \,|G_k(\Phi(u_n))|\,dx
    + \|f\|_m\left (\int_\Omega e^{m'H(u_n)} \,|G_k(\Phi(u_n))|^{m'}\,dx\right)^\frac 1{m'}
\end{align}
Denote for any $k$,
\begin{equation}\label{def_eta}
\eta(k)=\sup_{\{ \Phi(s)>k \}} \frac{e^{H(s)} }{(1+|\Phi(s)|)^{p-1}}\,.
\end{equation}
Since $\displaystyle\lim_{s\to \pm\infty}\Phi(s)=\pm\infty$, it is easy to verify from \eqref{ipo0} that
\begin{equation}\label{lim_eta}
\lim_{k\to +\infty}\eta(k)=0\,.
\end{equation}
Moreover as in \cite{PS}, by \eqref{lim_eta}, we obtain
\begin{align}\label{em'}
e^{m'H(u_n)} &\le \frac{e^{m'H(u_n)} }
{(1+|\Phi(u_n)|)^{m'(p-1)}
} (1+k+ |G_k(\Phi(u_n))   |)^{m'(p-1)}\\
 \nonumber&
\le  C\eta(k)^{m'}(k^{m'(p-1)}+|G_k(\Phi(u_n))   |^{m'(p-1)} )
\end{align}
for all $k\ge1$.
In analogous way we get
\begin{equation}\label{e}
e^{H(u_n)} \le
  C\eta(k)(k^{p-1}+|G_k(\Phi(u_n))   |^{p-1} )
\end{equation}
for all $k\ge1$.
By \eqref{est}, we deduce
\begin{multline}\label{est1}
 \lambda\int_\Omega|\nabla G_k(\Phi(u_n))|^p\, dx\\
 \le C\eta(k) \int_\Omega \frac{|G_k(\Phi(u_n))|}{u_n^{1-q}}(k^{p-1}+|G_k(\Phi(u_n))   |^{p-1} )\,dx\\
 +C\eta(k) \|f\|_m
 \left (\int_\Omega |G_k(\Phi(u_n))|^{m'}
 (k^{m'(p-1)}+|G_k(\Phi(u_n))   |^{m'(p-1)} )
 \,dx\right)^\frac 1{m'}
\end{multline}
Moreover since $\Phi$ is an increasing function, we deduce
\begin{align}\label{est2}
 \lambda\int_\Omega|\nabla G_k(\Phi(u_n))|^p\, dx\
 &\le
\frac{ C\eta(k)k^{p-1}}{ [\Phi^{-1}(k)]^{1-q}  } \int_\Omega |G_k(\Phi(u_n))| \, dx
\\
\nonumber&\\
 \nonumber&
+  \frac{ C\eta(k)}{ [\Phi^{-1}(k)]^{1-q}  } \int_\Omega |G_k(\Phi(u_n))|^p \, dx
 \nonumber&\\
 \nonumber&
+    C\eta(k) k^{p-1}\|f\|_m\left ( \int_\Omega |G_k(\Phi(u_n))|^{m'} \, dx \right)^\frac1{m'}\\
 \nonumber&\\
 \nonumber&
+    C\eta(k) \|f\|_m\left ( \int_\Omega |G_k(\Phi(u_n))|^{pm'} \, dx \right)^\frac1{m'}
\end{align}
for all $k\ge1$. The monotonicity of $\Phi$ also implies that $\Phi^{-1}(k)\ge 1$ for $k$ larger enough, so that we get rid of this coefficient for $k$ larger than certain $k'$.

\noindent Denote $A_k=\{ \Phi(u)\ge k  \}$. 
Now  we estimate each term in \eqref{est2} by using H\"older's inequality. The following inequality holds true:
\begin{equation*}
\int_\Omega |G_k(\Phi(u_n))| \, dx=\int_{A_k} |G_k(\Phi(u_n))| \, dx\le |A_k|^{1-\frac1{p^*}}\left(\int_{A_k} |G_k(\Phi(u_n))|^{p^*} \, dx\right)^{\frac1{p^*}}.
\end{equation*}
Thanks to these estimates and the Sobolev embedding theorem, \eqref{est2} becomes

\begin{align}\label{est3}
 &\lambda S\left (\int_\Omega| G_k(\Phi(u_n))|^{p^*}\, dx\right)^\frac p{p^*}\\
 \nonumber&\\
 \nonumber&\le
 C\eta(k)k^{p-1} |A_k|^{\frac 1{m'}-\frac 1{p^*}}|\Omega|^{1-\frac 1{m'}} \left (\int_\Omega| G_k(\Phi(u_n))|^{p^*}\, dx\right)^\frac 1{p^*}
\\
\nonumber&\\
 \nonumber&
\qquad + C\eta(k)  |\Omega|^{1-\frac p{p^*}} \left (\int_\Omega| G_k(\Phi(u_n))|^{p^*}\, dx\right)^\frac p{p^*}
 \nonumber&\\
 \nonumber&
\qquad  +C\eta(k) k^{p-1}  |A_k|^{\frac 1{m'}-\frac 1{p^*}} \|f\|_m \left (\int_\Omega| G_k(\Phi(u_n))|^{p^*}\, dx\right)^\frac 1{p^*}\
 \nonumber&\\
\nonumber&
\qquad    +C\eta(k) \|f\|_m  |\Omega|^{\frac 1{m'}-\frac p{p^*}}   \left (\int_\Omega| G_k(\Phi(u_n))|^{p^*}\, dx\right)^\frac p{p^*}
\end{align}
for all $k> k'$.
Since $\eta(k)$ goes to 0 when $k$ tends to $+\infty$, it yields that the second and the fourth terms on the right hand side can be absorbed by the left hand side. Then there exists $k_0>0$ such that, for $k>k_0$, we obtain:
\begin{equation}\label{est4}
  \left (\int_\Omega| G_k(\Phi(u_n))|^{p^*}\, dx\right)^\frac {p-1}{p^*}
 \le
 C\eta(k)k^{p-1}  |A_k|^{\frac 1{m'}-\frac 1{p^*}}
 \end{equation}
This is the same inequality as in \cite[Proposition 3.1]{PS}. So, following its procedure, we deduce that $ \{\Phi(u_n)\}_n$, and consequently $\{u_n\}_n$, is bounded in $L^\infty(\Omega)$. Therefore, there exists a positive constant $C_1>0$ satisfying $\| u_n\|_\infty\le C_1$ for all $n\in\N$.

Now we prove the  a priori estimates in $W_0^{1,p}(\Omega)$ given by \eqref{h01}.

\noindent  Consider  $\Psi(s)=s$ in Lemma \ref{canc} (2)  with $g (x, u_n)$ replaced by $T_n(g (x, u_n)) $. By \eqref{ell}, \eqref{g_pos},  \eqref{infty} and H\"older's inequality, we obtain
\begin{align}\label{est5}
 \lambda\int_\Omega|\nabla u_n|^p\, dx    &\le\int_\Omega T_n(g (x, u_n)) \, e^{H(u_n)} \,u_n\, dx +\int_\Omega f_n\,e^{H(u_n)} \,u_n\, dx
\\
 \notag  &\\
\notag    &\le\Lambda\int_\Omega |u_n|^{q} \, e^{H(u_n)} \, dx +\int_\Omega f\,e^{H(u_n)} \,u_n\, dx
\end{align}
Thus,
$$
\int_\Omega|\nabla u_n|^p\, dx\le \frac1\lambda  \Big[ \Lambda C_1^qe^{H(C_1)} |\Omega|+ e^{H(C_1)} C_1\|f\|_{L^1}\Big]
$$
 and estimate \eqref{h01} is proven.
  \end{proof}

  \medskip

\subsection{Strong convergence of $\nabla u_n$} 
 In this subsection we prove that  the sequence of the {\sl approximate solutions} $\{u_n\}_n$ and their gradients converge to a function $u$ and its gradient $\nabla u$ respectively.  Moreover we also prove that the  different terms appearing in equation \eqref{pd1} converge.
 
 For any fixed $n\in \N$, let $u_n\in W^{1,p}_0(\Omega)\cap L^\infty(\Omega)$ be a weak solution to problem \eqref{pd1}.
 As a consequence of Lemma \ref{apriori} we deduce that there exists a nonnegative function $u\in W_0^{1,p}(\Omega)\cap L^\infty(\Omega)$ such that, up to subsequences,
 \begin{align}
\label{d3pos}     \nabla u_n\rightharpoonup \nabla u \,,&\qquad \hbox{weakly in }L^p(\Omega; \R^N)\,,\\
\label{d4pos}   u_n\rightarrow u\,,&\qquad\hbox{strongly in }L^r(\Omega)\hbox{ for }1\le r<p^*\,,\\
\label{d5pos}    u_n(x)\rightarrow u(x)\,,&\qquad\hbox{a.e. in }\Omega\,.
  \end{align}
Actually, the $L^\infty$--estimate \eqref{infty} implies that
\begin{equation}
  \label{d6pos}   u_n\rightarrow u\,,\qquad\hbox{strongly in }L^r(\Omega)\hbox{ for }1\le r<+\infty\,.
\end{equation}

\begin{lemma}\label{strong convergence} 
({\sl Strong convergence of $\nabla u_n$} ).    Under the assumptions of Theorem \ref{main}, 
\begin{equation}\label{convgrad1}
\nabla u_n\rightarrow \nabla u\,, \quad \hbox{strongly in } (L^p(\Omega))^N\, 
 \end{equation}
 \begin{equation}\label{debole1}
\pmb{a}(x, u_n, \nabla u_n)\rightarrow \pmb{a}(x, u, \nabla u) \,,\quad\hbox{strongly in }L^{p'}(\Omega; \R^N)\,,
\end{equation}
\begin{equation}\label{convtermgrad}
b(x,u_n, \nabla u_n)\to b(x,u,\nabla u)\qquad\hbox{strongly in }L^1(\Omega)\,.
\end{equation}
\end{lemma}

\begin{proof}   We proceed to check all the conditions by dividing the proof in several steps.

\noindent {\sl Step 1. Strong convergence of the gradients.} In order to prove \eqref{convgrad1}, we  check that (see \cite[Lemma 5]{BMP3})
\begin{equation}\label{bmp}
\lim_{n\to+\infty}\int_\Omega
[\pmb{a}(x,u_n,\nabla u_n)-\pmb{a}(x,u_n,\nabla u)]\cdot \nabla (u_n-u) \, dx=0\,.
 \end{equation}
 Consider  $v= u_n-u  \in W^{1,p}_0(\Omega)$ in Lemma \ref{canc} (1)  to obtain
\begin{multline}\label{a1}
\int_\Omega e^{\hbox{sign\,}(u_n-u)H(u_n)} 
\pmb{a}(x,u_n,\nabla u_n)\cdot \nabla (u_n-u) \, dx
 \\
\le  \int_\Omega
T_n(g(x,u_n))  e^{\hbox{sign\,}(u_n-u)H(u_n)}  (u_n-u)  \, dx+ \int_\Omega
f_n e^{\hbox{sign\,}(u_n-u)H(u_n)}  (u_n-u)  \, dx\,.
  \end{multline}
Since $ u_n-u=(u_n-u) _+-(u_n-u) _-$ and $-(u_n-u) _-\le 0$ a.e. in $\Omega$, we obtain
 \begin{multline}\label{A1}
\int_\Omega e^{\hbox{sign\,}(u_n-u)H(u_n)}
\pmb{a}(x,u_n,\nabla u_n)\cdot \nabla (u_n-u)  \, dx
 \\
\le  \int_\Omega
T_n(g(x,u_n))  e^{\hbox{sign\,}(u_n-u)H(u_n)}  (u_n-u)_+  \, dx + \int_\Omega
f_n e^{\hbox{sign\,}(u_n-u)H(u_n)}  (u_n-u)  \, dx\,.
  \end{multline}
 Now let us analyze the following integral
 \begin{align}\label{a2}
\int_\Omega
& e^{\hbox{sign\,}(u_n-u)H(u_n)} [\pmb{a}(x,u_n,\nabla u_n)-\pmb{a}(x,u_n,\nabla u)]\cdot \nabla (u_n-u)\, dx
 \\[4mm] 
&\le -
 \int_\Omega
 e^{\hbox{sign\,}(u_n-u)H(u_n)}  \pmb{a}(x,u_n,\nabla u)\cdot \nabla (u_n-u)\, dx
\notag \\[4mm]
& \qquad +
 \int_\Omega
T_n(g(x,u_n))  e^{\hbox{sign\,}(u_n-u)H(u_n)}  (u_n-u)_+  \, dx 
\notag \\[4mm]
&\qquad +\int_\Omega
f_n e^{\hbox{sign\,}(u_n-u)H(u_n)}  (u_n-u)  \, dx
\notag \\[4mm]
&= I^1_n+I^2_n+I^3_n\,.\notag
\end{align}
Let us evaluate $I^1_n$. We prove
\begin{equation}\label{I1n}
\lim_{n\to +\infty}I^1_n =\lim_{n\to +\infty}\int_\Omega
 e^{\hbox{sign\,}(u_n-u)H(u_n)}  \pmb{a}(x,u_n,\nabla u)\cdot \nabla (u_n-u)\, dx=0
\end{equation}
 Indeed, first we split
 \begin{multline*}
   \int_\Omega
 e^{\hbox{sign\,}(u_n-u)H(u_n)}  \pmb{a}(x,u_n,\nabla u)\cdot \nabla (u_n-u)\, dx \\
   =\int_\Omega
 e^{H(u_n)}  \pmb{a}(x,u_n,\nabla u)\cdot \nabla (u_n-u)_+\, dx
 -\int_\Omega
 e^{-H(u_n)}  \pmb{a}(x,u_n,\nabla u)\cdot \nabla (u_n-u)_-\, dx
 \end{multline*}
 We remark that, owing to \eqref{d3pos},
 \begin{equation}\label{ppos}
   \nabla (u_n-u)_+\rightharpoonup 0\quad\hbox{weakly in }L^{p}(\Omega; \R^N)\,.
 \end{equation}
 In fact, if $\varphi\in C_0^\infty(\Omega)$, then
 \[\int_\Omega \frac{\partial (u_n-u)_+}{\partial x_i}\varphi \, dx=-\int_\Omega(u_n-u)_+\frac{\partial \varphi}{\partial x_i}\, dx\]
 tends to 0 for all $i=1, \dots, N$, by \eqref{d4pos}.

 Since
 \begin{align*}
e^{H(u_n)}&  |\pmb{a}(x,u_n,\nabla u)|\\[4mm]
&\le  e^{H(C_1)} (a_0 |\nabla u|^{p-1}+a_1|u_n|^{p-1}+a_2 )
 \,,\qquad \hbox{a.e. in }\Omega\,,
  \end{align*}
   it follows from \eqref{d4pos} that the right hand side converges in $L^{p'}(\Omega)$, so that the left hand side is equiintegrable.
Therefore by \eqref{d5pos} and Vitali's convergence theorem we deduce
\begin{equation}\label{conv1}
e^{H(u_n)}  \pmb{a}(x,u_n,\nabla u)\rightarrow e^{H(u)}\pmb{a}(x,u,\nabla u)\,, \qquad \hbox{strongly in } L^{p'}(\Omega)^N \
\end{equation}
Combining \eqref{conv1} and \eqref{ppos}, we infer that
\[\lim_{n\to +\infty}\int_\Omega
 e^{H(u_n)}  \pmb{a}(x,u_n,\nabla u)\cdot \nabla (u_n-u)_+\, dx=0\]
 Analogously,
\[\lim_{n\to +\infty}\int_\Omega
 e^{-H(u_n)}  \pmb{a}(x,u_n,\nabla u)\cdot \nabla (u_n-u)_-\, dx=0\]
and \eqref{I1n} is proven.

\noindent Let us evaluate $I_n^2$. By the growth condition on $g$ \eqref{g_pos}, for any fixed $\delta>0$, we get:
\begin{multline*}
I^2_n\le  \Lambda\int_\Omega e^{\hbox{sign\,}(u_n-u)H(u_n)}  (u_n-u)_+  u_n^{q-1}\, dx
 \\[4mm]
 \le
 \Lambda\delta^{q-1}\int_{\{u_n\ge \delta  \}}e^{\hbox{sign\,}(u_n-u)H(u_n)}  (u_n-u)_+  \, dx
 +
  \Lambda\int_{\{u_n\le \delta  \}\cap \{u_n\ge u  \} }e^{H(u_n)}  u_n^q \, dx
\\[4mm]
\le
 \Lambda  \delta^{q-1} e^{H(C_1)}
 \int_\Omega |u_n-u|  \, dx
 +
\Lambda  \delta^{q} e^{H(C_1)} |\Omega|
\end{multline*}
It yields
\begin{equation}\label{I2n}
\limsup_{n\to\infty} I^2_n\le \Lambda  \delta^{q} e^{H(C_1)} |\Omega|
\end{equation}
for all $\delta>0$, so that $\lim_{n\to\infty}I^2_n=0$.

Finally we evaluate $I^3_n$. By \eqref{d4pos}, our assumption of summability on $f$ and the $L^\infty$ estimate, we have
\begin{equation}\label{I3n}
\lim_{n\to +\infty}I^3_n =\lim_{n\to +\infty}\int_\Omega
 e^{\hbox{sign\,}(u_n-u)H(u_n)}  f(u_n-u)\, dx=0
\end{equation}
By \eqref{a2}, combining \eqref{I1n}, \eqref{I2n} and \eqref{I3n}, we get
$$
\limsup_{n\to+\infty}\int_\Omega
e^{\hbox{sign\,}(u_n-u)H(u_n)} [\pmb{a}(x,u_n,\nabla u_n)-\pmb{a}(x,u_n,\nabla u)]\cdot \nabla (u_n-u)\, dx
\le0\,.
$$
Since by \eqref{ell}
$$
e^{\hbox{sign\,}(u_n-u)H(u_n)} [\pmb{a}(x,u_n,\nabla u_n)-\pmb{a}(x,u_n,\nabla u)]\cdot \nabla (u_n-u)\ge 0
$$
we deduce
 \begin{equation}\label{a3}
\lim_{n\to+\infty}\int_\Omega
e^{\hbox{sign\,}(u_n-u)H(u_n)} [\pmb{a}(x,u_n,\nabla u_n)-\pmb{a}(x,u_n,\nabla u)]\cdot \nabla (u_n-u)\, dx
=0\,.
 \end{equation}
Therefore by \eqref{a3}, we have
\begin{align}\label{afin}
&\lim_{n\to+\infty}\int_\Omega
 [\pmb{a}(x,u_n,\nabla u_n)-\pmb{a}(x,u_n,\nabla u)]\cdot \nabla (u_n-u)\, dx
 \\[4mm]
 &\notag= \lim_{n\to+\infty}
 \int_\Omega
 e^{\hbox{sign\,}(u_n-u)H(u_n)}   e^{-\hbox{sign\,}(u_n-u)H(u_n)} \pmb{a}(x,u_n,\nabla u)\cdot \nabla (u_n-u)\, dx
\\[4mm]
 &\notag\le  e^{H(C_1)}
\lim_{n\to+\infty}  \int_\Omega
 e^{\hbox{sign\,}(u_n-u)H(u_n)}  \pmb{a}(x,u_n,\nabla u)\cdot \nabla (u_n-u)\, dx=0
\,.
\end{align}
This proves \eqref{bmp} and therefore \eqref{convgrad1}.

\medskip

\noindent {\sl Step 2. Strong convergence of  $\pmb{a}(x, u_n, \nabla u_n)$ and $b(x, u_n, \nabla u_n)$}
 A straightforward conseguence of \eqref{convgrad1} is that, up to subsequences,
 \begin{equation}\label{A5}
    \nabla u_n\rightarrow \nabla u \,,\qquad\hbox{a.e. in }\Omega\,.
  \end{equation}
Easy consequences of the pointwise convergence of the gradients are
 \[\pmb{a}(x, u_n, \nabla u_n)\rightarrow \pmb{a}(x, u, \nabla u)\,,\qquad\hbox{a.e. in }\Omega,\]
 and
 \[b(x, u_n, \nabla u_n)\rightarrow b(x, u, \nabla u)  \,,\qquad\hbox{a.e. in }\Omega\,.\]
Furthermore, the strong convergence of the gradients \eqref{convgrad1} jointly with \eqref{d6pos} imply that the sequence
\[a_0|\nabla u_n|^{p-1}+a_1|u_n|^{p-1}+a_2\]
is equi-integrable. So, our hypothesis \eqref{crescitaa}, gives the equiintegrability of $\pmb{a}(x, u_n, \nabla u_n)$ and, by Vitali's Theorem, \eqref{debole1} follows.
On the other hand, the $L^\infty$--estimate leads to the boundedness of $h(u_n)$. Hence, it follows from \eqref{crescitaa} and the strong convergence of the gradients that the sequence $b(x, u_n, \nabla u_n)$ is equiintegrable. Applying again Vitali's Theorem, we get  \eqref{convtermgrad}.
Additionally, we also obtain
\begin{equation}\label{binl1}
b(x,u,\nabla u)\in L^1(\Omega).
\end{equation}
\end{proof} 
\medskip

\subsection{Existence: proof of Theorem \ref{main}}\label{ex_pos}

In this subsection, we prove that the function $u$ is a weak solution to problem \eqref{P} according to Definition \ref{defsol}.

\noindent Since we have proved  \eqref{binl1},    condition \eqref{b} is satisfied. Therefore we proceed to check the other conditions  in Definition \ref{defsol}.

\noindent {\sl Step 1. $u$ satisfies  \eqref{4.00} }

\noindent
Let us consider the weak solution $u_n$  to the approximate problem \eqref{pd1}. For any fixed  $k>0$ we take  $v=G_k(u_n^{1-q})$ in Lemma \ref{canc} (1) and so
\begin{align*}
    (1-q)\lambda&\int_{\{u_n^{1-q}>k\}} e^{ H(u_n)}\ \pmb{a}(x,u_n,\nabla  u_n)\cdot \nabla u_n \, u_n^{-q} \, dx \\
    &\le \int_{\Omega} e^{ H(u_n)} T_n(g(x,u_n))G_k(u_n^{1-q})\, dx +\int_\Omega e^{ H(u_n)} f_n\, G_k(u_n^{1-q}) \, dx\,.
\end{align*}
By assumption on $g$ \eqref{g_pos}, since $u_n$ is a nonnegative function, we have:
$$
T_n(g(x, u_n))\,G_k(u_n^{1-q})\le g(x, u_n)\,u_n^{1-q} \le \Lambda\,,
$$
a.e. in $\{u_n^{1-q}\ge k\}$. 

 By ellipticity condition \eqref{ell}, Remark 2.1 and estimate \eqref{infty}, we get
\begin{align*}
    (1-q)\lambda&\int_{\{u_n^{1-q}>k\}}e^{ H(u_n)} u_n^{-q}|\nabla u_n|^p\, dx \\
    &\le C \int_{\Omega}T_n(g(x,u_n))G_k(u_n^{1-q})\, dx +C\int_\Omega f_n\, G_k(u_n^{1-q})\, dx\\
    &\le C\Lambda |\Omega | +  \,C\int_\Omega f\, G_k(u_n^{1-q})\, dx\\
     &\le C\Lambda |\Omega | +   \,C\|u_n\|_\infty \int_\Omega f\, dx \le C \,.
\end{align*}
Since $e^{ H(u_n)} \ge 1$,  
we obtain
\[\int_{\{u_n^{1-q}>k\}} |u_n|^{-q}|\nabla u_n|^p\, dx\le C\,.\]
So, owing to the Fatou Lemma,
\begin{equation}\label{4.1}
   \int_{\{u^{1-q}>k\}} |u|^{-q}|\nabla u|^p\, dx
\le C\,.
\end{equation}
Now we let $k$ go to $0$ on the left-hand side,  by monotone convergence Theorem,
$$
\lim_{k\to0}\int_{\{u^{1-q}>k\}} u^{-q}|\nabla u|^pdx=\int_{\{u>0\}} u^{-q}|\nabla u|^pdx<+\infty\,.
$$
Applying \cite[Lemma 2.5]{GPS}, we obtain that $u^{1-\frac{q}p}\in W^{1,p}_0(\Omega)$ and
$$
\int_{\{u>0\}} u^{-q}|\nabla u|^pdx=\int_{\Omega} u^{-q}|\nabla u|^pdx\,,
$$
so that $u^{-q}|\nabla u|^p\in L^1(\Omega)$.

\medskip

\noindent {\sl Step 2. $u$ satisfies  \eqref{4.01_pos}}

\noindent Consider  $v\in W^{1,p}_0(\Omega)\cap L^\infty(\Omega)$, $v\ge 0$ a.e. in $\Omega$ as test function in \eqref{pd1}, we get
\begin{equation*}
\int_\Omega \pmb{a}(x,u_n,\nabla  u_n)\cdot \nabla v\, dx
    =\int_\Omega b(x,u_n,\nabla u_n)v\, dx +\int_\Omega T_n\left(g (x, u_n )\right)v\, dx +\int_\Omega f_n v\, dx \,.
\end{equation*}

Therefore passing to the limit for $n$ which goes to $+\infty$,
by  \eqref{d5pos}, \eqref{convtermgrad} and Fatou's lemma, we get
\begin{multline}\label{finale}
\lim_{n\to+\infty}\int_\Omega \pmb{a}(x,u_n,\nabla  u_n)\cdot \nabla v\, dx\\
 -\lim_{n\to +\infty}\left (\int_\Omega b(x,u_n,\nabla u_n)v\, dx +\int_\Omega f_n v\, dx\right)\\
    \ge \int_\Omega g (x, u)v\, dx  \,.
\end{multline}
that is
\begin{multline}\label{finale}
\int_\Omega \pmb{a}(x,u,\nabla  u)\cdot \nabla v\, dx
 -\int_\Omega b(x,u,\nabla u)v\, dx  - \int_\Omega f v\, dx  \ge \int_\Omega g (x, u)v\, dx \,.
\end{multline}
This yields the conclusion for $v\ge0$. The general case follows from the decomposition $v=v_+-v_-$.

\medskip

\noindent {\sl Step 3. Proof of  \eqref{4.0} by passing to the limit in the approximate problems}.

Let   $\varphi$ be any nonnegative function belonging to $ W^{1,p}(\Omega)\cap L^\infty(\Omega)$. Taking
$T_k(u_n)\varphi$ as test function in \eqref{pd1} and disregarding a nonnegative term, we have
\begin{multline}\label{step3_1}
\int_\Omega T_k(u_n)\left (\pmb a(x,u_n,\nabla  u_n)\cdot\nabla\varphi\right )\, dx
    \le
    \int_\Omega b(x,u_n, \nabla u_n) T_k(u_n)\varphi\, dx\\
    +\int_\Omega T_n\left(g (x, u_n )\right)T_k(u_n)\varphi\, dx +\int_\Omega f_n T_k(u_n)\varphi\, dx \,.
\end{multline}

Now we let $n$ go to $+\infty$ in the inequality \eqref{step3_1}. On the left-hand side we use the strong convergence of $\big\{ \pmb a(x,u_n,\nabla  u_n) \big\}_n$ in  $L^{p'}(\Omega)$, \eqref{debole1}, the pointwise convergence of $u_n$ \eqref{d5pos} and Lebesgue's dominated convergence theorem in order to obtain
$$
\lim_{n\to +\infty} \int_\Omega T_k(u_n)\left (\pmb a(x,u_n,\nabla  u_n)\cdot\nabla\varphi\right )\, dx=
\int_\Omega T_k(u)\left (\pmb a(x,u,\nabla  u)\cdot\nabla\varphi\right )\, dx\,.
$$
In analogous way,  we evaluate the first term on the right-hand side \eqref{step3_1}. We apply the strong convergence of $\big\{ b(x,u_n,\nabla  u_n) \big\}_n$ in $L^1(\Omega)$ given by \eqref{binl1}, the pointwise convergence of $u_n$ \eqref{d5pos} and Lebesgue's dominated convergence theorem in order to obtain
$$
\lim_{n\to +\infty}  \int_\Omega b(x,u_n, \nabla u_n) T_k(u_n)\varphi\, dx= \int_\Omega b(x,u, \nabla u) T_k(u)\varphi\, dx\,.
$$

Concerning the second term on the right-hand side of \eqref{step3_1}, we observe that
\begin{multline*}
  T_n\left(g (x, u_n )\right)T_k(u_n) \\
  =T_n\left(g (x, u_n )\right)T_k(u_n)\big|_{\{u_n\le k\}}+T_n\left(g (x, u_n )\right)T_k(u_n)\big|_{\{u_n> k\}}\\
  \le \Lambda u_n^{q-1} u_n\big|_{\{u_n\le k\}}+\Lambda k u_n^{q-1}\big|_{\{u_n> k\}}
  \le   \Lambda  k^q\,.
\end{multline*}
 Therefore  we can apply Lebesgue's dominated convergence Theorem and we get
$$
\lim_{n\to +\infty}  \int_\Omega T_n\left(g (x, u_n )\right)T_k(u_n)\varphi\, dx= \int_\Omega g (x, u )T_k(u)\varphi\, dx\,.
$$
Finally it is easy to verify that 
$$
\lim_{n\to +\infty}  \int_\Omega f_n T_k(u_n)\varphi\, dx= \int_\Omega f T_k(u)\varphi\, dx.
$$
Hence, passing to the limit for $n$ which goes to $+\infty$ in \eqref{step3_1}, we get
\begin{multline*}
    \int_\Omega T_k(u)\left( \pmb a(x,u,\nabla  G_k(u))\cdot\nabla\varphi \right )\, dx \\
    \le
    \int_\Omega b(x,u, \nabla  u) T_k(u)\varphi\, dx
    +\int_\Omega g (x, u)T_k(u)\varphi\, dx +\int_\Omega f T_k(u)\varphi\, dx\\
    \le \int_\Omega b(x,u, \nabla  u) T_k(u)\varphi\, dx
    +k\int_\Omega g (x, u)\varphi\, dx +k\int_\Omega f \varphi\, dx \,.
\end{multline*}
Dividing by $k$ and letting $k$ go to $0$,  it follows that
\begin{gather*}
    \int_{\{u\ne0\}} \left (\pmb a(x,u,\nabla  u)\cdot\nabla\varphi\right )\, dx  \\
    \le
    \int_{\{u\ne0\}} b(x,u, \nabla  u) \varphi\, dx
    + \int_\Omega g(x,u)\,\varphi\, dx +\int_\Omega f \varphi\, dx
\end{gather*}
holds true. As a consequence of Stampacchia's Theorem (cf. \cite{S}), we obtain
\begin{gather*}
    \int_{\Omega} \left (\pmb a(x,u,\nabla  u)\cdot\nabla\varphi\right )\, dx  \\
    \le
    \int_{\Omega} b(x,u, \nabla  u) \varphi\, dx
    + \int_\Omega g(x,u)\,\varphi\, dx +\int_\Omega f \varphi\, dx
\end{gather*}
Since \eqref{finale} yields the reverse inequality, we conclude
\begin{gather*}
    \int_{\Omega} \left (\pmb a(x,u,\nabla  u)\cdot\nabla\varphi\right )\, dx  \\
    =
    \int_{\Omega} b(x,u, \nabla  u) \varphi\, dx
    + \int_\Omega g(x,u)\,\varphi\, dx +\int_\Omega f \varphi\, dx
\end{gather*}
for all $\varphi\ge0$. The general case is now straightforward.   \qed

\section{Existence result for $m= N/p$}

The main result of the section concerns existence of nonnegative weak solutions to problem \eqref{P}   when the datum $f$ is  an element of the Lebesgue space $L^m(\Omega)$, with $m= N/p$. Its statement is the following.

\begin{theorem}\label{mainlimit}
Assume \eqref{crescitaa} - \eqref{ipf_pos} with
$$
f\in L^\frac Np(\Omega)
$$
and
\begin{equation}\label{ipo1}
\lim_{s\to \pm \infty}\frac{e^{H(s)}}{(1+\Phi(s))^{p-1}}=0\,.
\end{equation}
Then problem \eqref{P}  has at least a weak solution $u$ such that $\Phi(u)\in W_0^{1,p}(\Omega)\cap L^r(\Omega)$, and hence $u\in W_0^{1,p}(\Omega)\cap L^r(\Omega)$, for all $1\le r<\infty$.
\end{theorem}
As in the previous case, we consider the  approximate problems \eqref{pd1}, which for any fixed $n$, has at least a nonnegative bounded weak solution $u_n\in W^{1,p}_0(\Omega)\cap L^\infty(\Omega)$.
We begin by  proving a priori estimates for weak solutions $u_n$ which implies the existence of a limit function $u$ which is proven to be a weak solution to problem \eqref{P}.

\subsection{A priori estimates when $m=N/p$}\label{s3dex2}
In this subsection we  prove that the sequence of {\sl approximate solutions} $\{u_n\}_n$ satisfies a priori estimates in $L^r(\Omega)$, for any $r>1$, and in $W^{1,p}_0(\Omega)$. By these estimates we  deduce that $u_n$, up to subsequences, converges to a limit function $u$ which is the sought solution.

\begin{lemma}\label{apriori1}
({\sl Estimates in $L^r(\Omega)$ for all $1\le r<\infty$ and in $W^{1,p}_0(\Omega)$}).  For any fixed $n\in \N$, let $u_n\in W^{1,p}_0(\Omega)\cap L^\infty(\Omega)$ be a weak solution to problem \eqref{pd1}. Under the assumptions of Theorem \ref{main}, the following estimates hold true:
\begin{equation}\label{infty1}
\|u_n\|_{L^r(\Omega)}\le  C_3\qquad 1\le r<\infty\,, 
\end{equation}
\begin{equation}\label{h02}
\|\nabla u_n\|_{L^p(\Omega)}\le  C_4\,,
\end{equation}
where $ C_3$, $ C_4$ are positive constants which only depend on $|\Omega|$, $N$, $m$, $p$, $\|f\|_{L^m}$, $\lambda$  and on $r$,
but do  not depend on $n$.

Furthermore, it is also satisfied
\begin{equation}\label{conv:03a}
  \lim_{k\to\infty}\sup_{n\in\N}\int_\Omega|\nabla G_k(u_n)|^p\, dx=0\,.
\end{equation}
\end{lemma}
\begin{proof} For any $\gamma\ge1$, consider $\Psi(s)=\Phi(s)^\gamma$  in Lemma \ref{canc} (2). Then, 
 since $e^\frac{pH(u_n)}{p-1} \ge1$, we get
\begin{multline*}
\lambda \int_\Omega \gamma  \Phi(u_n)^{\gamma-1}|\nabla \Phi(u_n)|^pdx
  \le \lambda \int_\Omega \gamma e^\frac{pH(u_n)}{p-1} \Phi(u_n)^{\gamma-1}|\nabla u_n|^pdx\\
  \le \int_\Omega e^{H(u_n)} \Phi(u_n)^{\gamma} g(x, u_n)\, dx+\int_\Omega e^{H(u_n)} \Phi(u_n)^{\gamma} f\, dx
\end{multline*}
This inequality and Sobolev's imbedding Theorem lead to
\begin{multline}\label{inec:01}
\left(\int_\Omega \Phi(u_n)^{\frac{\gamma+p-1}{p}p^*}dx\right)^{\frac{p}{p^*}}\le C\int_\Omega|\nabla \Phi(u_n)^{\frac{\gamma+p-1}{p}}|^pdx\\
\le C\int_{\{\Phi(u_n)>k\}}e^{H(u_n)} \Phi(u_n)^{\gamma}g(x,u_n)\, dx
+C\int_{\{\Phi(u_n)>k\}}e^{H(u_n)} \Phi(u_n)^{\gamma}f\, dx\\
+ C\int_{\{\Phi(u_n)\le k\}}e^{H(u_n)} \Phi(u_n)^{\gamma}g(x,u_n)\, dx
+C\int_{\{\Phi(u_n)\le k\}}e^{H(u_n)} \Phi(u_n)^{\gamma}f\, dx
\end{multline}
where $k\ge 1$.
We handle the integrals over $\{\Phi(u_n)>k\}$ employing the function $\eta$ defined in \eqref{def_eta}. Notice that, since $k\ge1$, it results $1+\Phi(u_n)\le 2\Phi(u_n)$  on the set $\{\Phi(u_n)>k\}$ and therefore, we deduce
\begin{align*}
  \int_{\{\Phi(u_n)>k\}}e^{H(u_n)}& \Phi(u_n)^{\gamma}g(x,u_n)\, dx\\
  &\le \int_{\{\Phi(u_n)>k\}}\eta(k)(1+\Phi(u_n))^{p-1} \Phi(u_n)^{\gamma}g(x,u_n)\, dx \\
  &\le \int_{\{\Phi(u_n)>k\}}\eta(k)2^{p-1} \Phi(u_n)^{\gamma+p-1}g(x,u_n)\, dx
\end{align*}
 Thus, we get
\begin{align}\label{1}
\int_{\{\Phi(u_n)>k\}}e^{H(u_n)} \Phi(u_n)^{\gamma}&g(x,u_n)\, dx\le C\eta(k) \int_{\{\Phi(u_n)>k\}}\Phi(u_n)^{\gamma+p-1}u_n^{q-1}\, dx\\
&\notag\\
&\notag\le C\frac{\eta(k)}{[\Phi^{-1}(k)]^{1-q}}\int_{\{\Phi(u_n)>k\}}\Phi(u_n)^{\gamma+p-1}\, dx \notag\\
&\notag\\
&\notag\le C\frac{\eta(k)}{[\Phi^{-1}(k)]^{1-q}}|\Omega|^{p/N}\left(\int_{\Omega}\Phi(u_n)^{(\gamma+p-1)\frac{p^*}p}\, dx\right)^{\frac p{p^*}}\,.
\end{align}
 Moreover by H\"older inequality, we get
\begin{align}\label{2}
  \int_{\{\Phi(u_n)>k\}}e^{H(u_n)} \Phi(u_n)^{\gamma}f\, dx
 & \le C \eta(k) \int_{\{\Phi(u_n)>k\}}\Phi(u_n)^{\gamma+p-1} f\, dx\\
 &\notag\le C\eta(k) \|f\|_{N/p}\left(\int_{\Omega}\Phi(u_n)^{(\gamma+p-1)\frac{p^*}p}\, dx\right)^{\frac p{p^*}}\,.
\end{align}
Arguing as in Lemma \ref{apriori},  since $\eta (k)$ goes to zero when $k$ tends to $+\infty$,  the terms in the right-hand side of \eqref{1} and \eqref{2}  can be absorbed by the left-hand side of \eqref{inec:01}. 
Hence, there exists $k$ larger enough such that \eqref{inec:01} becomes
\begin{multline}\label{inec:02}
\left(\int_\Omega \Phi(u_n)^{\frac{\gamma+p-1}{p}p^*}dx\right)^{\frac{p}{p^*}}\le C\int_\Omega|\nabla \Phi(u_n)^{\frac{\gamma+p-1}{p}}|^pdx\\
\le  C\int_{\{\Phi(u_n)\le k\}}e^{H(u_n)} \Phi(u_n)^{\gamma}g(x,u_n)\, dx
+C\int_{\{\Phi(u_n)\le k\}}e^{H(u_n)} \Phi(u_n)^{\gamma}f\, dx\,.
\end{multline}
Now observe that the right hand side is bounded. Indeed, since by \eqref{g_pos} and \eqref{fisuu}, it follows that
\begin{multline*}
e^{H(u_n)} \Phi(u_n)^{\gamma}g(x,u_n)\le \Lambda u_n^{q-1}e^{H(u_n)} \Phi(u_n)^{\gamma}\\= \Lambda e^{H(u_n)} u_n^{\gamma+q-1}\left(\frac{\Phi(u_n)}{u_n}\right)^\gamma
\le \Lambda e^{H(u_n)} u_n^{\gamma+q-1}e^{\gamma\frac{H(u_n)}{p-1}}\,,
\end{multline*}
Thus, we get
\[\int_{\{\Phi(u_n)\le k\}}e^{H(u_n)} \Phi(u_n)^{\gamma}g(x,u_n)\, dx\le C |\Omega| e^{H(\Phi^{-1}(k))} \Phi^{-1}(k)^{\gamma+q-1}e^{\gamma\frac{H(\Phi^{-1}(k))}{p-1}}.\]
The remainder term in \eqref{inec:02} is obviously bounded,  that is
$$
\int_{\{\Phi(u_n)\le k\}}e^{H(u_n)} \Phi(u_n)^{\gamma}f\, dx \le 
C e^{H(\Phi^{-1}(k))} \Phi^{-1}(k)^{\gamma}\|f\|_{L^1}\,.
$$
Therefore, it follows from \eqref{inec:02} that
\[\left(\int_\Omega \Phi(u_n)^{\frac{\gamma+p-1}{p}p^*}dx\right)^{\frac{p}{p^*}}\le C\int_\Omega|\nabla \Phi(u_n)^{\frac{\gamma+p-1}{p}}|^pdx\le C\,,
\]
for all $\gamma\ge1$. 
 Hence by the arbitrary of $\gamma$ and therefore of  $\frac{\gamma+p-1}{p}p^*\ge 1$,
 the sequence $\{\Phi(u_n)\}_n$ is bounded in every $L^r(\Omega)$ such that $1\le r <\infty$ and, taking $\gamma=1$, in $W_0^{1,p}(\Omega)$. We conclude that the same features hold for $\{u_n\}_n$.

Condition \eqref{conv:03a} follows from
\begin{equation*}
  \lim_{k\to\infty}\sup_{n\in\N}\int_\Omega|\nabla G_k(\Phi(u_n))|^p\, dx=0\,.
\end{equation*}
and it yields from performing the following computations (with some $\gamma>1$):
\begin{equation*}
\int_\Omega|\nabla G_k(\Phi(u_n))|^p\, dx\le
\int_\Omega \frac{\Phi(u_n)^{\gamma-1}}{k^{\gamma-1}}|\nabla G_k(\Phi(u_n))|^p\, dx
\le \frac{C}{k^{\gamma-1}}\,.
\end{equation*}
\end{proof}

\begin{remark}\rm
In the previous proof, we have seen that the sequence $\{\Phi(u_n)\}_n$ is bounded in every $L^r(\Omega)$,  with $1\le r<+\infty$. As a consequence of assumption \eqref{ipo1}, we deduce that
$\{e^{H(u_n)}\}_n$ is also bounded in every $L^r(\Omega)$  with $1\le r<+\infty$.
\end{remark}

\medskip

\subsection{Strong convergence of $\nabla u_n$} \label{sec:convgrad}
 In this subsection we prove that  the sequence of the {\sl approximate solutions} $\{u_n\}_n$ and their gradients converge to a function $u$ and its gradient $\nabla u$ respectively.  Moreover  we prove that every term in equation \eqref{pd1} converges.
 
For any fixed $n\in \N$, let $u_n\in W^{1,p}_0(\Omega)\cap L^\infty(\Omega)$ be a weak solution to problem \eqref{pd1}. 
By Lemma \ref{apriori1}  there exists a nonnegative function $u\in W_0^{1,p}(\Omega)\cap L^r(\Omega)$ for all $1\le r<+\infty$ such that, up to subsequences, the convergences  in \eqref{d3pos},  \eqref{d4pos} and \eqref{d5pos} hold true. In this limit case, we also obtain the convergence appearing in \eqref{d6pos}. Furthermore, the pointwise convergence allows us to obtain the strong convergence of $e^{H(u_n)}$ to $e^{H(u)}$ in any $1\le r<\infty$.

\begin{lemma}\label{strong convergence2}
({\sl Strong convergence of $\nabla u_n$} ).   Under the assumptions of Theorem \ref{mainlimit}, 
\begin{equation}\label{convgrad2}
\nabla u_n\rightarrow \nabla u\,, \quad \hbox{strongly in } (L^p(\Omega))^N\, 
 \end{equation}
 \begin{equation}\label{debole2}
\pmb{a}(x, u_n, \nabla u_n)\rightarrow \pmb{a}(x, u, \nabla u) \,,\quad\hbox{strongly in }L^{p'}(\Omega; \R^N)\,,
\end{equation}
\begin{equation}\label{strongb}
b(x,u_n, \nabla u_n)\to b(x,u,\nabla u)\qquad\hbox{strongly in }L^1(\Omega)\,.
\end{equation}
\end{lemma}

\begin{proof} We proceed to check all the conditions by dividing the proof in several steps.

\noindent {\sl Step 1. Strong convergence of the gradients.} 
We explicitly point out that we cannot apply the same proof of the previous case because now we do not have an $L^\infty$--bound for $\{u_n\}_n$.

\noindent As in the previous case, in order to prove \eqref{convgrad2}, we are proving that (recall  \cite[Lemma 5]{BMP3})
\begin{equation}\label{conv:02}
  \lim_{n\to+\infty}\int_\Omega
 [\pmb{a}(x,u_n,\nabla u_n)-\pmb{a}(x,u_n,\nabla u)]\cdot \nabla (u_n-u)\, dx=0\,
\end{equation}
holds true.
To this aim we write
$$
\nabla (u_n-u)= \nabla T_k(u_n-T_h(u))+\nabla G_k(u_n-T_h(u))+\nabla (T_h(u)-u)\,,
$$
for certain $k$ and $h$, with $k>h$, to be chosen.
Hence
\begin{multline}\label{conv:02a}
\int_\Omega
 [\pmb{a}(x,u_n,\nabla u_n)-\pmb{a}(x,u_n,\nabla u)]\cdot \nabla (u_n-u)\, dx=\\
\int_\Omega
 [\pmb{a}(x,u_n,\nabla u_n)-\pmb{a}(x,u_n,\nabla u)]\cdot \nabla T_k(u_n-T_h(u))\, dx \\
 +
 \int_\Omega
 [\pmb{a}(x,u_n,\nabla u_n)-\pmb{a}(x,u_n,\nabla u)]\cdot   \nabla (T_h(u)-u)\, dx\\
 +
 \int_\Omega
 [\pmb{a}(x,u_n,\nabla u_n)-\pmb{a}(x,u_n,\nabla u)]\cdot   \nabla G_k(u_n-T_h(u))\, dx
 \,
\end{multline}
Let us begin by proving the following equality
\begin{equation}\label{conv:01}
  \lim_{n\to+\infty}\int_\Omega
 [\pmb{a}(x,u_n,\nabla u_n)-\pmb{a}(x,u_n,\nabla u)]\cdot \nabla T_k(u_n-T_h(u))\, dx=0
\end{equation}
We may take $v=  T_k(u_n-T_h(u))    \in W^{1,p}_0(\Omega)$ in Lemma \ref{canc} (1). Then we proceed as in the previous case. Actually we integrate over $\{|u_n|\le k+h\}$ and we may argue as above replacing $C_1$ with $k+h$. This yields \eqref{conv:01}.

Let us evaluate the second integral on the right-hand side of \eqref{conv:02a}. Fix $\epsilon>0$. Taking into account that the sequence $\{u_n\}_n$ is bounded in $W_0^{1,p}(\Omega)$, it follows from condition \eqref{crescitaa} that $\{\pmb{a}(x, u_n, \nabla u_n)\}_n$ is bounded in $L^{p'}(\Omega)^N$. Hence, there exists $h>0$ satisfying
\begin{multline}\label{conv:03}
\int_\Omega
 |[\pmb{a}(x,u_n,\nabla u_n)-\pmb{a}(x,u_n,\nabla u)]\cdot \nabla (T_h(u)-u)|\, dx\\
 \le\Big(\|\pmb{a}(x,u_n,\nabla u_n)\|_{p'}+\|\pmb{a}(x,u_n,\nabla u)\|_{p'}\Big)\left(\int_{\{u>h\}}|\nabla u|^p\right)^{1/p}
 <\epsilon/3\quad\forall n\in\N\,.
\end{multline}
Having fixed $h$, we  determine $k$. To this end,  notice that

\begin{multline}
\int_\Omega |[\pmb{a}(x,u_n,\nabla u_n)-\pmb{a}(x,u_n,\nabla u)]\cdot \nabla G_k(u_n-T_h(u))|\, dx\\
=\int_{\{|u_n-T_h(u)|>k\}} |[\pmb{a}(x,u_n,\nabla u_n)-\pmb{a}(x,u_n,\nabla u)]\cdot \nabla G_k(u_n-T_h(u))|\, dx\\
=\int_{\{u_n>k-h\}} |[\pmb{a}(x,u_n,\nabla u_n)-\pmb{a}(x,u_n,\nabla u)]\cdot \nabla G_k(u_n-T_h(u))|\, dx\,.
\end{multline}
Therefore by growth condition \eqref{crescitaa} and a priori estimates \eqref{h02}, we have
\begin{multline*}
  \int_\Omega |[\pmb{a}(x,u_n,\nabla u_n)-\pmb{a}(x,u_n,\nabla u)]\cdot \nabla G_k(u_n-T_h(u))|\, dx \\
  \le \Big(\|\pmb{a}(x,u_n,\nabla u_n)\|_{p'}+\|\pmb{a}(x,u_n,\nabla u)\|_{p'}\Big )\left[\left(\int_{\{u_n>k-h\}}|\nabla u_n|^p\right)^{1/p}+\left(\int_{\{u_n>k-h\}}|\nabla u|^p\right)^{1/p}\right]\\
  \le C\left[\left(\int_{\{u_n>k-h\}}|\nabla u_n|^p\right)^{1/p}+\left(\int_{\{u_n>k-h\}}|\nabla u|^p\right)^{1/p}\right]
\end{multline*}
Taking into  account \eqref{conv:03a}, we may find $k$ such that
\begin{equation}\label{conv:03bis}\int_\Omega |[\pmb{a}(x,u_n,\nabla u_n)-\pmb{a}(x,u_n,\nabla u)]\cdot \nabla G_k(u_n-T_h(u))|\, dx<\epsilon/3\quad\forall n\in\N.
\end{equation}
Combining \eqref{conv:02a}, \eqref{conv:01}, \eqref{conv:03} and \eqref{conv:03bis}, we conclude that \eqref{conv:02} holds.
\medskip

\noindent {\sl Step 2. Strong convergence of  $\pmb{a}(x, u_n, \nabla u_n)$ } A straightforward conseguence of \eqref{convgrad2} is that, up to subsequences,
 \begin{equation}\label{A5}
    \nabla u_n\rightarrow \nabla u \,,\qquad\hbox{a.e. in }\Omega\,.
  \end{equation}
Moreover, we also infer
\[b(x, u_n, \nabla u_n)\rightarrow b(x, u, \nabla u) \,, \quad \hbox{a.e. in } \Omega\]
and
\[\pmb{a}(x, u_n, \nabla u_n)\rightarrow \pmb{a}(x, u, \nabla u) \,, \quad \hbox{a.e. in } \Omega\,.\]
Now, as in the proof of Step 2 of Lemma \ref{strong convergence}, it follows from \eqref{crescitaa} and \eqref{convgrad2} that
\begin{equation}\label{debole2}
\pmb{a}(x, u_n, \nabla u_n)\rightarrow \pmb{a}(x, u, \nabla u) \,,\quad\hbox{strongly in }L^{p'}(\Omega; \R^N)\,.
\end{equation}

\noindent {\sl Step 3. Strong convergence of $b(x,u_n,\nabla u_n)$ to $b(x,u,\nabla u)$}
Now we prove \eqref{strongb}. Since no $L^\infty$--estimate is available, we are not able to prove that $\{h(u_n)\}_n$ is bounded, so that the   proof given in the previous section is not possible.

\noindent Consider the function
\[\Xi(s)=\int_0^sh(\sigma)\chi_{\{\sigma>k\}}d\sigma\qquad (k\ge1)\]
and note that $\Xi(s)\le H(s)\chi_{\{s>k\}}$ holds. Taking $\Psi(s)=\Xi(s)$ in Lemma \ref{canc} (2), we obtain
\begin{multline}\label{crucial-des}
  \lambda \int_{\{u_n>k\}}e^{H(u_n)}h(u_n)|\nabla u_n|^pdx\\
  \le \int_\Omega e^{H(u_n)}\Xi(u_n) g(x,u_n)\, dx+\int_\Omega e^{H(u_n)}\Xi(u_n) f\, dx\\
  \le \Lambda \int_{\{u_n>k\}} e^{H(u_n)}H(u_n) k^{q-1}\, dx+\int_{\{u_n>k\}}e^{H(u_n)}H(u_n) f\, dx
\end{multline}
We note that $k^{q-1}\le1$ and use the fact  that $\{e^{H(u_n)}\}_n$, and hence $\{H(u_n)\}_n$, is bounded in any $L^r(\Omega)$, $1\le r<\infty$, to get 
\[\int_{\{u_n>k\}}e^{H(u_n)}h(u_n)|\nabla u_n|^pdx\le  C\|(1+f)\chi_{\{u_n>k\}}\|_{L^m(\Omega)},\]
so that the right hand side tends to 0 uniformly on $n$.
 Therefore,
\begin{equation*}
  \lim_{k\to\infty}\sup_{n\in\N}\int_{\{u_n>k\}} e^{H(u_n)}h(u_n)|\nabla u_n|^p\, dx=0
\end{equation*}
 and, since $e^{H(u_n)}\ge 1$,
\begin{equation}\label{conv:05}
  \lim_{k\to\infty}\sup_{n\in\N}\int_{\{u_n>k\}} h(u_n)|\nabla u_n|^p\, dx=0
\end{equation}
The main consequence is that the sequence $\{h(u_n)|\nabla u_n|^p\}_n$ is equiintegrable. Indeed, if $E\subset \Omega$, then
\begin{multline}\label{inec:03}
  \int_E h(u_n)|\nabla u_n|^pdx\\
  = \int_{E\cap \{u_n\le k\}} h(u_n)|\nabla u_n|^pdx+\int_{E\cap \{u_n>k\}} h(u_n)|\nabla u_n|^pdx\\
  \le \big[\max_{s\in [0,k]}h(s)\big] \int_E |\nabla u_n|^pdx+\int_{\{u_n>k\}} h(u_n)|\nabla u_n|^pdx
\end{multline}
Now let $\epsilon>0$.
Keeping in mind \eqref{conv:05} and choosing $k$ large enough, we may obtain that
$$\int_{\{u_n>k\}} h(u_n)|\nabla u_n|^pdx<\epsilon/2$$
for all $n\in\N$. Fixed $k$, we may use the strong convergence of gradients to deduce that there exists $\delta>0$ such that, for any set $E$ having $|E|<\delta$,
$$\big[\max_{s\in [0,k]}h(s)\big] \int_E |\nabla u_n|^pdx<\epsilon/2$$ for all $n\in\N$. Therefore by \eqref{inec:03},
$|E|<\delta$ implies $\int_E h(u_n)|\nabla u_n|^pdx<\epsilon$ for all $n\in\N$, which provides the equiintegrability of the sequence $\{h(u_n)|\nabla u_n|^p\}_n$. Then by \eqref{crescitab} we conclude that the sequence $\{b(x, u_n, \nabla u_n)\}_n$ is equiintegrable.
Applying Vitali's Theorem, \eqref{strongb} follows. 
\end{proof}

\subsection{Existence: proof of Theorem \ref{mainlimit}}\label{ex_pos2} We prove that the function $u$  is a weak solution to problem \eqref{P} according to Definition \ref{defsol}. The proof of Theorem \ref{mainlimit} prooceds exactly as the proof of Theorem \ref{main} in subsection \ref{ex_pos}. We explicitely remark that the strong convergence in \eqref{strongb} obviously implies \eqref{b}. \qed

\medskip

\section{Existence result for $\frac{Np}{Np-N+p}\le m<\frac Np$}

The main result of the section, concerning existence of nonnegative weak solutions to problem \eqref{P} when the datum $f$ is  an element of the Lebesgue space $L^m(\Omega)$ with  $\frac{Np}{Np-N+p}\le m<\frac Np$, is stated as follows:
\begin{theorem}\label{main_unbounded}
Assume \eqref{ell} - \eqref{ipf_pos} with
$$
f\in L^m(\Omega)\,,\qquad \frac{Np}{N(p-1)+p}\le m< \frac Np\,.
$$
Moreover assume that there exist a constant $0<\theta<\frac{p^*}{pm'}$ and constants $0<M_1\le M_2$  satisfying
\begin{equation}\label{ipo2}
M_1\le \frac{e^{H(s)}}{(1+\Phi(s))^{(p-1)\theta}}\le M_2\qquad \forall s\ge0\,.
\end{equation}
Then problem \eqref{P}  has at least a weak solution such that
$u\in W_0^{1,p}(\Omega)\cap L^\frac{Nm(p-1)}{N-pm}(\Omega)$.
\end{theorem}
As in the previous cases, we consider problems \eqref{pd1}, which for any fixed $n$, has at least a nonnegative bounded weak solution $u_n\in W^{1,p}_0(\Omega)\cap L^\infty(\Omega)$ and we  prove a priori estimates for these approximate solutions  $u_n$.
\subsection{A priori estimates}\label{s3dex1}

In this subsection we prove
that the sequence of {\sl approximate solutions} $\{u_n\}_n$ satisfies a priori estimates in $L^\frac{Nm(p-1)}{N-pm}(\Omega)$ and in $W^{1,p}_0(\Omega)$.
 We point out that $\frac{Nm(p-1)}{N-pm}$ tends to $\infty$ as $m\to\frac Np$ and so there is no solution of continuity with the estimates of the previous section. Observe, in addition, that $m=\frac{Np}{Np-N+p}$ yields an estimate in $L^{p^*}(\Omega)$, as expected.

\begin{lemma}\label{apriori2}
({\sl Estimates in $L^\frac{Nm(p-1)}{N-pm}(\Omega)$ and $W^{1,p}_0(\Omega)$}).   For any fixed $n\in \N$, let $u_n\in W^{1,p}_0(\Omega)\cap L^\infty(\Omega)$ be a weak solution to problem \eqref{pd1}. Under the assumptions of Theorem \ref{main_unbounded}, the following estimates hold true:
\begin{equation}\label{infty-r}
\|u_n\|_{L^\frac{Nm(p-1)}{N-pm}(\Omega)}\le  C_5\,,
\end{equation}
\begin{equation}\label{h03}
\|\nabla u_n\|_{L^p(\Omega)}\le  C_6\,,
\end{equation}
where $ C_5$, $ C_6$ are positive constants which only depend on $|\Omega|$, $N$, $m$, $p$, $\|f\|_{L^m}$, $\lambda$,
but do  not depend on $n$.

Furthermore, it also holds
\begin{equation}\label{conv:07}
  \lim_{k\to\infty}\sup_{n\in\N}\int_\Omega|\nabla G_k(u_n)|^p\, dx=0\,.
\end{equation}
\end{lemma}

\begin{proof}
For any $r>p-1$, consider
$\Psi(s)=(1+\Phi(s)^{r-p}\Phi(s)$
in Lemma  \ref{canc} (2). Since  $\Psi'(u_n)\ge \min \{r+1-p, 1\}(1+\Phi(u_n))^{r-p}\Phi'(u_n)$, we get
\begin{multline}\label{V0}
 \int_\Omega  (1+\Phi(u_n))^{r-p}|\nabla \Phi(u_n)|^p dx
  =C\int_\Omega  e^\frac{pH(u_n)}{p-1}(1+\Phi(u_n))^{r-p}|\nabla u_n|^pdx\\
  \le C\int_\Omega e^{H(u_n)} (1+\Phi(u_n))^{r-p}\Phi(u_n)g(x, u_n)\, dx\\
  +C\int_\Omega e^{H(u_n)} (1+\Phi(u_n))^{r-p}\Phi(u_n)f_n\, dx\,.
\end{multline}
This inequality and Sobolev's imbedding Theorem lead to
\begin{align}\label{inec:01a}
&\left(\int_\Omega [(1+\Phi(u_n))^{\frac{r}{p}}-1]^{p^*}dx\right)^{\frac{p}{p^*}}
\le C\int_\Omega (1+\Phi(u_n))^{r-p}|\nabla \Phi(u_n)|^pdx\\
    &\notag\\
&\le C\int_{\{\Phi(u_n)>k\}}e^{H(u_n)} (1+\Phi(u_n))^{r-p+1}g(x, u_n)\, dx\notag\\
    &\notag\\
&\qquad + C\int_{\{\Phi(u_n)\le k\}}e^{H(u_n)}(1+\Phi(u_n))^{r-p}\Phi(u_n)g(x, u_n)\, dx\notag\\
    &\notag\\
&\notag \qquad +
C\int_{\Omega}e^{H(u_n)} (1+\Phi(u_n))^{r-p+1}f\, dx\,,
\end{align}
where $k\ge 1$.
Now we evaluate  the  integral over $\{\Phi(u_n)>k\}$ in \eqref{inec:01a}. Since $k\ge1$, by assumption \eqref{ipo2},  we have
\begin{align}\label{V1}
  \int_{\{\Phi(u_n)>k\}}e^{H(u_n)} &(1+\Phi(u_n))^{r-p+1}g(x, u_n)\, dx\\
  &\notag\\
  &\le M_2 \Lambda \int_{\{\Phi(u_n)>k\}}(1+\Phi(u_n))^{r-(p-1)(1-\theta)}  u_n^{q-1}\, dx \notag\\
    &\notag\\
&\le \frac{M_2 \Lambda}{[\Phi^{-1}(k)]^{1-q}}\int_{\{\Phi(u_n)>k\}}(1+\Phi(u_n))^{r-(p-1)(1-\theta)}  \, dx\notag\\
    &\notag\\
&\notag\le \frac{M_2 \Lambda}{[\Phi^{-1}(k)]^{1-q}}|\Omega|^{1-\frac 1m}\left(\int_{\Omega}
(1+\Phi(u_n))^{[r-(p-1)(1-\theta)]{m'}}
\, dx\right)^{\frac 1{m'}}\,.
\end{align}
Next we analyze the second integral over $\{\Phi(u_n)\le k\}$ in \eqref{inec:01a}.
Taking into account \eqref{g_pos}, \eqref{fisuu} and \eqref{ipo2}, we get
\begin{align*}
e^{H(u_n)}(1+\Phi(u_n))^{r-p}\Phi(u_n)g(x, u_n)&\le
\Lambda e^{H(u_n)}(1+\Phi(u_n))^{r-p}u_n^{q}\left(\frac{\Phi(u_n)}{u_n}\right)\\[2mm]
&\le \Lambda e^{H(u_n)}(1+\Phi(u_n))^{r-p}u_n^{q} e^{\frac{H(u_n)}{p-1}}
\end{align*}
Consequently, we get
\begin {multline}\label{V2}
\int_{\{\Phi(u_n)\le k\}}e^{H(u_n)}(1+\Phi(u_n))^{r-p}\Phi(u_n)g(x, u_n)\, dx\\
\le
\Lambda |\Omega| e^{H(\Phi^{-1}(k))}(1+k)^{r-p} \Phi^{-1}(k)^{q} e^{\frac{H(\Phi^{-1}(k))}{p-1}}\,.
\end{multline}
Regarding the third term on the right hand side of \eqref{inec:01a}, we apply \eqref{ipo2} and the H\''older inequality to obtain
\begin{multline}\label{V3}
\int_{\Omega}e^{H(u_n)} (1+\Phi(u_n))^{r-p+1}f\, dx\le
M_2\int_{\Omega} (1+\Phi(u_n))^{r-(p-1)(1-\theta)}f\, dx
\\
\le \|f\|_{L^m(\Omega)} \left(\int_{\Omega}
(1+\Phi(u_n))^{[r-(p-1)(1-\theta)]{m'}}
\, dx\right)^{\frac 1{m'}}
\end{multline}
Owing to \eqref{V1}, \eqref{V2} and \eqref{V3}, inequality \eqref{inec:01a} becomes
\begin{multline}\label{V4}
\left(\int_\Omega [(1+\Phi(u_n))^{\frac{r}{p}}-1]^{p^*}dx\right)^{\frac{p}{p^*}}\\
\le C+C
\left(\int_{\Omega}
(1+\Phi(u_n))^{[r-(p-1)(1-\theta)]{m'}}
\, dx\right)^{\frac 1{m'}}
\end{multline}
Now it follows from $ m< \frac Np$ that $m'>\frac{p^*}{p}$. This fact allows us to choose $r$ such that
\begin{equation}\label{V5}
r=\frac{(p-1)(1-\theta)m'}
{m'-\frac{p^*}{p}}\,,
\end{equation}
so that
$$
[r-(p-1)(1-\theta)]{m'}=\frac{rp^*}{p}=\frac{Nm(p-1)(1-\theta)}{N-pm}\,.
$$
Therefore, we infer from \eqref{V4} that there exists a constant $C>0$ satisfying
\begin{equation}\label{V6}
\int_\Omega (1+\Phi(u_n))^{\frac{Nm(p-1)(1-\theta)}{N-pm}}dx
\le C
\end{equation}
and, going back to \eqref{V0}, that
\begin{equation}\label{V7}
 \int_\Omega  (1+\Phi(u_n))^{r-p}|\nabla \Phi(u_n)|^p dx
\le C
\end{equation}
for all $n\in\N$.

To go from these estimates on $\{\Phi(u_n)\}_n$ to estimates on $\{u_n\}_n$, we first apply \eqref{ipo2} to get
\begin{equation}\label{V8}
M_1^{\frac1{p-1}}\le \frac{\Phi'(s)}{(1+\Phi(s))^{\theta}}
\end{equation}
and so
\[
M_1^{\frac1{p-1}}s\le \frac1{1-\theta} (1+\Phi(s))^{1-\theta}
\]
holds for all $s\in\R$.
On the one hand, this last inequality, jointly with \eqref{V6}, gives an estimate of $\{u_n\}_n$ in $L^ \frac{Nm(p-1)}{N-pm}(\Omega)$, so that \eqref{infty-r} is proven.  On the other, \eqref{V8} and \eqref{V7} imply
\begin{multline*}
M_1^{\frac{p}{p-1}}\int_\Omega (1+\Phi(u_n))^{r-p(1-\theta)}|\nabla u_n|^pdx
\le \int_\Omega (1+\Phi(u_n))^{r-p} \Phi'(u_n)^{p}|\nabla u_n|^pdx
\\
=  \int_\Omega(1+\Phi(s))^{r-p}|\nabla \Phi(u_n)|^p dx\le C\,.
\end{multline*}
Since $\displaystyle r-p(1-\theta)=\frac{p^*-m'}{m'-\frac{p^*}p}(1-\theta)\ge0$,  the estimate \eqref{h03} follows.

\noindent It remains to check \eqref{conv:07}. Having already determined $r$ by \eqref{V5}, we now choose
$\Psi(s)=(1+\Phi(G_k(s))^{r-p}\Phi(G_k(s))$
in Lemma  \ref{canc} (2), with $k\ge1$, getting
\begin{align*}
 \int_\Omega  e^{H(u_n)}e^\frac{H(G_k(u_n))}{p-1} &(1+\Phi(G_k(u_n)))^{r-p}|\nabla G_k(u_n)|^pdx\\
  &\le \int_\Omega e^{H(u_n)} (1+\Phi(G_k(u_n)))^{r-p}\Phi(G_k(u_n))g(x, u_n)\, dx\\
  &\qquad +\int_\Omega e^{H(u_n)} (1+\Phi(G_k(u_n))^{r-p}\Phi(G_k(u_n))f_n\, dx\\
  &\le \int_{\{u_n>k\}}e^{H(u_n)} (1+\Phi(u_n))^{r-p+1}g(x, u_n)\, dx\\
  &\qquad + \int_{\{u_n>k\}}e^{H(u_n)} (1+\Phi(u_n))^{r-p+1}f\, dx\,.
\end{align*}
Arguing as above, we deduce that
\begin{multline*}
 \int_\Omega  |\nabla G_k(u_n)|^pdx
  \le C \int_{\{u_n>k\}} (k^{q-1}+f)e^{H(u_n)} (1+\Phi(u_n))^{r-p+1}\, dx\\
  \le C \|(1+f) \chi_{\{u_n>k\}}\|_{L^{m}(\Omega)}
  \left(\int_{\Omega}  (1+\Phi(u_n))^{[r-(p-1)(1-\theta)]{m'}}\, dx\right)^{\frac1{m'}}\\
  \le C \|(1+f) \chi_{\{u_n>k\}}\|_{L^{m}(\Omega)}\,.
\end{multline*}
due to $k^{q-1}\le 1$.
Since the right hand side tends to 0 as $k\to\infty$, condition \eqref{conv:07} follows.
\end{proof}

 \begin{remark}\rm
We point out that we have obtained \eqref{V6}, an estimate on $\{\Phi(u_n)\}_n$, which leads, thanks to \eqref{ipo2}, to an estimate on $\{e^{H(u_n)}\}_n$, namely:
\begin{equation}\label{V9}
\int_\Omega \left( e^{H(u_n)}\right)^{\frac{Nm(1-\theta)}{\theta(N-pm)}}dx\le C\,.
\end{equation}
\end{remark}

\medskip
\subsection{Existence: proof of Theorem \ref{main_unbounded}}\label{ex_pos2}
By Lemma \ref{apriori2}  there exists a nonnegative function $u\in W_0^{1,p}(\Omega)\cap L^\frac{Nm(p-1)}{N-pm}(\Omega)$  such that, up to subsequences, conditions \eqref{d3pos},  \eqref{d4pos} and \eqref{d5pos} hold true. Actually, and owing to \eqref{infty-r}, we have that the strong convergence in \eqref{d4pos} holds for every $1\le r<\frac{Nm(p-1)}{N-pm}$. Moreover, the pointwise convergence \eqref{d5pos}, \eqref{V9} and
\[m'<\frac{Nm(1-\theta)}{\theta(N-pm)}\]
imply that
\begin{equation}\label{V10}
 e^{H(u_n)}\to e^{H(u)}\quad \hbox{in }L^{m'+\epsilon}(\Omega)\,.
\end{equation}
for some $\epsilon>0$ small enough.

As in the previous case the proof of Theorem \ref{main_unbounded} needs
\begin{enumerate}
\item the strong convergence of $\nabla u_n$ to $\nabla u$ in $L^p(\Omega)^N$
\item the strong convergence of $b(x,u_n,\nabla u_n)$ to $b(x,u,\nabla u)$ in $L^1(\Omega)$
\end{enumerate}

\noindent As far as the strong convergence of $\nabla u_n$ concerns, the proof proceeds exactly as in Step 1 of the proof of Lemma \ref{strong convergence2}, bearing in mind that \eqref{conv:07} holds.

In an analogous way the proof of the strong convergence of $b(x,u_n,\nabla u_n)$ to $b(x,u,\nabla u)$ proceeds as in Step 3 of the proof of Lemma \ref{strong convergence2}. Just replace, on the right hand side of \eqref{crucial-des}, the fact that $\{e^{H(u_n)}\}_n$ is bounded in any $L^r(\Omega)$, $r<\infty$ with the fact that $\{e^{H(u_n)}\}_n$ is bounded in $L^{m'+\epsilon}(\Omega)$. Then it is enough to perform the following inequalities over the set $\{u_n>k\}$:
\[
e^{H(u_n)}H(u_n)\le \frac{2m'}\epsilon e^{(1+(\epsilon/2m'))H(u_n)}\le  \frac{2m'}\epsilon e^{(1+\epsilon/m')H(u_n)}\frac1{e^{\epsilon H(k)/m'}}
\]
wherewith
\[
e^{m'H(u_n)}H(u_n)^{m'}\chi_{\{u_n>k\}}\le  C e^{(m'+\epsilon)H(u_n)}\frac1{e^{\epsilon H(k)}}\chi_{\{u_n>k\}}\,.
\]

\noindent Finally the conclusion of the proof of Theorem \ref{main_unbounded} follows the argument given in Subsection \ref{ex_pos}.  \qed

\section*{Acknowledgements}
The authors are grateful to David Arcoya for useful discussions and suggestions. The authors would like to thank University of Campania  ``L. Vanvitelli", University of Naples Federico II and University of Valencia for supporting some visiting appointments and  their kind hospitality.

\section*{Declarations}

\noindent{\bf {Ethical Approval}} Not applicable.

\noindent {\bf{Funding}}  The research of A. Ferone was partially supported by  Italian MIUR through research project PRIN 2017 “Qualitative and quantitative aspects of nonlinear PDEs.”. The research of
A. Mercaldo was partially supported by  Italian MIUR through research project PRIN 2017 “Direct and inverse problems for partial differential equations: theoretical aspects and applications” and PRIN 2022 "Partial differential equations and related geometric-functional inequalities". The research of
S. Segura de Le\'on is partially supported by Grant PID2022-136589NB-I00 founded
by MCIN/AEI/10.13039/501100011033 as well as by Grant RED2022-134784-T founded
by MCIN/AEI/10.13039/501100011033.

\noindent{\bf{Availability of data and materials}} All data generated or analysed during this study are included in this article.

\end{document}